\documentclass[12pt]{amsart}
\usepackage{graphicx}
\usepackage[T1]{fontenc}
\usepackage{amssymb,amsmath,amsfonts, amsthm}
\usepackage{tikz}
\usetikzlibrary[calc,patterns] 

\newcommand{\R}{\mathbb{R}}

\newcommand{\SL}{\mathrm{SL}}

\newtheorem{theorem}{Theorem}[section]
\newtheorem{proposition}[theorem]{Proposition}

\newtheorem{lemma}[theorem]{Lemma}

\newtheorem{example}[theorem]{Example}

\newtheorem*{thmNN}{Theorem}

\theoremstyle{definition}

\theoremstyle{remark}
\newtheorem{rem}[theorem]{Remark}

\title{Systoles in translation surfaces}
\author{Corentin Boissy, Slavyana Geninska}

\address{Institut de Math\'ematiques de Toulouse \\ 
Universit\'e Toulouse 3 \\
118 route de Narbonne \\
31062 Toulouse, France
}
\email{geninska@math.univ-toulouse.fr}
\email{corentin.boissy@math.univ-toulouse.fr}

\begin{document}
\begin{abstract}
For a translation surface, we define the relative systole to be the length of the shortest saddle connection. We give a characterization of the maxima of the systole function on a stratum, and give a family of examples providing local but nonglobal maxima on each stratum of genus at least three.  We further study the relation between (locally) maximal values of the systole function and the number of shortest saddle connections.
\end{abstract}

\maketitle
\section{Introduction}
This paper deals with flat metrics defined by Abelian differentials on compact Riemann surfaces (\emph{translation surfaces}). Such flat metrics have conical singularities of angle $(k+1)2\pi$, where $k$ is the order of the zero of the corresponding Abelian differential. 
A stratum of the moduli space of Abelian differential corresponds to translation surfaces that share the same combinatorics of zeroes, possibly including marked points.

A saddle connection on a translation surface is a geodesic joining two singularities (possibly the same) and with no singularity in its interior. 
A sequence of area one translation surfaces in a stratum leaves any compact set if and only if the length of the shortest saddle connection tends to zero.
The set of translation surfaces with short saddle connections and compactification issues of strata are related to dynamics and counting problems on translation surfaces and have been widely studied in the last 30 years (see for instance \cite{KMS,EMZ,EKZ}).

In this paper, we are interested in the opposite problem: we study surfaces that are as far as possible from the boundary and that would represent the ``core'' of a stratum. For a translation surface, we define the  \textit{relative systole} $\mathrm{Sys}(S)$ to be the length of the shortest saddle connection of $S$. Our primary goal is to study global and local maxima of the function $\mathrm{Sys}$ when restricted to area one translation surfaces. Note that our definition is different from the ``true systole''   \emph{i.e.} shortest closed curve that has been studied by Judge and Parlier  in \cite{JP}. In the rest of the paper, for simplicity, if not mentioned otherwise, the term ``systole'' will mean ``relative systole''.

This kind of question appears also in other contexts. Maxima of the systole function for moduli spaces of hyperbolic surfaces, where the systole is the length of the shortest closed geodesic, has been studied by various authors, for instance Bavard \cite{Ba}, Schmutz Schaller \cite{pS94}, or more recently Balacheff, Makover and  Parlier \cite{BMP}. 
A related question is the maximal number of geodesics realizing the systole, the so called kissing number, see for instance Schmutz Schaller \cite{pSS97}, Fanoni and Parlier \cite{FaPa}.

In the context of area one translation surfaces, while the characterization of global maxima for $\mathrm{Sys}$ seems to have been known for some time in the mathematic community, the existence of local maxima was unknown. We provide explicit examples of local maxima that are not global in each stratum with genus $g=2$ with marked points or $g\geq 3$. We also study the relation between the (locally) maximal values of the function $\mathrm{Sys}$ and the (locally) maximal number of shortest saddle connections.

The paper is organized as follows. In Section~2, we give some general background on translation surfaces.

In Section~3, we study global maxima of the function $\mathrm{Sys}$ for area one translation surfaces. We prove the following theorem (see Theorem~\ref{th:max:sys}):

\begin{thmNN}
Let $S$ be a genus $g\geq 1$ translation surface of area one and $r>0$ singularities or marked points. Then 
$$\mathrm{Sys}(S) \leq \left(\frac{\sqrt{3}}{2}(2g-2+r)\right)^{-\frac{1}{2}}.$$
The equality is obtained if and only if $S$ is built with equilateral triangles with sides saddle connections of length $\mathrm{Sys}(S)$. Such surface exists in any connected component of any stratum.
\end{thmNN}

This result was independently proven recently by Judge and Parlier  \cite{JP} for surfaces with one singularity: the authors are interested in shortest closed curves but their proof should work in any strata in our context. 

In Section~4, we study local maxima of the function $\mathrm{Sys}$ that are not global. With the help of explicit examples we prove the following result which is Theorem~\ref{th:loc:max:in:each:stratum} in the text.
\begin{thmNN}
Each stratum of area one surfaces with genus $g=2$ with marked points or $g\geq 3$ contains local maxima of the function $\mathrm{Sys}$ that are not global.
\end{thmNN}
The examples are obtained by considering surfaces that decompose into equilateral triangles and regular hexagons, with some further conditions (see Theorem~\ref{th:local:maxima} for a precise statement).


In the last section, we study the relation between (locally) maximal values of the function $\mathrm{Sys}$ and the (locally) maximal number of shortest saddle connections. We call a surface \emph{rigid} if it corresponds to a local maximum of the number of shortest saddle connections. While the connection is clear for global maxima (see Proposition~\ref{prop:max:number}), the situation is more complex for the local maxima. The examples that we provide for local maxima of the function $Sys$ are rigid. Even more, a surface that is a local maximum of the function $Sys$ and that decomposes into equilateral triangles and regular hexagons must be rigid (Proposition~\ref{prop:loc:block}). However, rigid surfaces are not necessarily local maxima (see Proposition~\ref{prop:block:nonloc}).

The authors thank Carlos Matheus for pointing out a small mistake in the first version of the paper, and the anonymous referee for the improvement suggestions.

\section{Background}
A \emph{translation surface} is a (real, compact, connected) genus $g$ surface $S $ with a translation atlas \emph{i.e.} a triple $(S,\mathcal U,\Sigma)$ such that $\Sigma$ is a finite subset of $S$ (whose elements are called {\em singularities}) and $\mathcal U = \{(U_i,z_i)\}$ is an atlas of $S \setminus \Sigma$ whose transition maps are translations of $\mathbb{C}\simeq \mathbb{R}^2$. We will require that  for each $s\in \Sigma$, there is a neighborhood of $s$ isometric to a Euclidean cone whose total angle is a multiple of $2\pi$.  One can show that the holomorphic structure on $S\setminus \Sigma$ extends to $S$ and that the holomorphic 1-form $\omega=dz_i$ extends to a holomorphic $1-$form on $X$ where  $\Sigma$ corresponds to the zeroes of $\omega$ and maybe some marked points. We usually call $\omega$ an \emph{Abelian differential}.  A zero of $\omega$ of order $k$ corresponds to a singularity of angle $(k+1)2\pi$. By a slight abuse of notation, we authorize the order of a zero to be 0, in this case it corresponds to a regular marked point.
A saddle connection is a geodesic segment joining two singularities (possibly the same) and with no singularity in its interior. Integrating $\omega$ along the saddle connection we get a complex number. Considered as a planar vector, this complex number represents the affine holonomy vector of the saddle connection. In particular, its Euclidean length is the modulus of its holonomy vector. 

For $g \geq 1$, we define the moduli space of Abelian differentials $\mathcal{H}_g$ as the moduli space of pairs $(X,\omega)$ where $X$ is a genus $g$ (compact, connected) Riemann surface and $\omega$ non-zero holomorphic $1-$form defined on $X$. The term moduli space means that we identify the points $(X,\omega)$ and $(X',\omega')$ if there exists an analytic isomorphism $f:X \rightarrow X'$ such that 
$f^*\omega'=\omega$.
The group $\SL(2,\mathbb R)$ naturally acts on the moduli space of translation surfaces by post composition on the charts defining the translation structures.

One can also see a translation surface obtained as a polygon (or a finite union of polygons) whose sides come by pairs, and for each pair, the corresponding segments are parallel and of the same length. These parallel sides are glued together by translation and we assume that this identification preserves the natural orientation of the polygons. In this context, two translation surfaces are identified in the moduli space of Abelian differentials if and only if the corresponding polygons can be obtained from each other by cutting and gluing and preserving the identifications. Also, the $\SL(2,\mathbb{R})$ action in this representation is just the natural linear action on the polygons. 	

The moduli space of Abelian differentials is stratified by the  combinatorics of the zeroes; we will denote by $\mathcal{H}(k_1,\ldots ,k_r)$ the stratum of $\mathcal{H}_g$, where $\sum_i k_i=2g-2$, consisting of (classes of) pairs $(X,\omega)$ such that $\omega$ has exactly $r$ zeroes, of order $k_1,\dots ,k_r$. 
This space is  (Hausdorff) complex analytic (see for instance \cite{Ma82, Vee82, Vee90}). We often restrict to the subset $\mathcal{H}_1(k_1,\dots ,k_r)$ of \emph{area one} surfaces.
 Local coordinates for a stratum of Abelian differentials are obtained by integrating the holomorphic 1--form along a basis of the relative homology $H_1(S,\Sigma;\mathbb{Z})$, where $\Sigma$ denotes the set of conical singularities of $S$.

\section{Maximal systole}
We recall that the systole $\mathrm{Sys}(S)$ of a translation surface $S$ is the length of the shortest saddle connection of $S$. The aim of this section is to prove Theorem~\ref{th:max:sys} which characterizes translation surfaces of area one with maximal systole. One key tool is Delaunay triangulation.
 
Let $S$ be a translation surface. A \textit{Delaunay triangulation} $S$ is a triangulation of $S$ such that the vertices are singularities, the 1-cells (the sides of the triangles) are saddle connections and, for a 2-cell (triangle) $T$ of the triangulation, the circumcircle of any representative $\tilde T$ of the universal covering does not have any singularity in its interior.

In Section~4 of \cite{MS91} Masur and Smillie prove the existence of Delaunay triangulations for every translation surface $S$.

\begin{lemma}
\label{L:ShortSaddleDelaunay}
All shortest saddle connections of $S$ are 1-cells in every Delaunay triangulation of $S$.
\end{lemma}
\begin{proof}
Let $\sigma$ be a saddle connection that is not included in a Delaunay triangulation $\mathcal{T}$. Denote by $P,Q$ the extremities of $\sigma$. Let $T \in \mathcal{T}$ be the triangle in $\mathcal{T}$ with $P$ as a vertex and containing a subsegment of $\sigma$.  Let $P',P''$ be the other vertices of $T$ (see Figure~\ref{fig:lemma:3.1}).

\begin{figure}
\begin{tikzpicture}[scale=1]

\draw (0,0) circle (2);
\draw [very thick] (1.2,-1.6) arc (-53:70:2) node {$\bullet$} node[above right] {$P''$};
\draw (1.2,-1.6) node{$\bullet$} node[below] {$P'$};
\draw (-2,0)  node{$\bullet$} node[left] {$P$};
\draw (-2,0) -- (2.4,0) node[midway, below] {$\sigma$} node {$\bullet$} node[right] {$Q$};
\draw[dashed] (-2,0) -- (1.2,-1.6);
\draw[dashed] (-2,0) -- (0.75,1.9);
\draw (-1.8,1.3) node {$c$} ;

\end{tikzpicture}
\caption{Illustration of Lemma~\ref{L:ShortSaddleDelaunay}}
\label{fig:lemma:3.1}
\end{figure}

Consider the circumcircle $c$ of $T$, and the open arc of $P'P''$ that does not contain $P$. Each chord of $c$ joining $P$ to an element of this arc is of length strictly greater than $\min(d(P,P'),d(P,P''))\geq \mathrm{Sys}(S)$. One of these chords is in the direction of $\sigma$ and since there is no singularity in the interior of $c$, this chord is a subsegment of $\sigma$. Therefore, $\sigma$ is not a shortest saddle connection.

\end{proof}

The first statement of the following lemma is needed for the proof of the next theorem. The second statement will be useful for Theorem~\ref{th:loc:max:in:each:stratum}. 
\begin{lemma}\label{lemma:decomp:triang:equil}
Let $\mathcal{C}\subset \mathcal{H}(k_1,\dots ,k_r)$ be a connected component of a stratum of abelian differentials with $k_1,\dots ,k_r\geq 0$.  
\begin{enumerate}
\item There exists in $\mathcal{C}$ a surface $S$ that decomposes into equilateral triangles with sides saddle connections. 
\item Furthermore, for each  $i, j$  we can find such a surface with a side of  an equilateral triangle being a saddle connection joining a singularity of degree $k_i$ to a singularity of degree $k_j$, with the convention that the two singularities are different if $i\neq j$ and equal if $i=j$.
\end{enumerate} 
\end{lemma}

\begin{proof}
We first prove (1). By Lemma~18 in \cite{KoZo} there exists in each connected component of each stratum a surface with a horizontal one cylinder decomposition. Up to a shear transformation that creates a vertical saddle connection, such surface can be described as a rectangle with the two vertical sides identified that  correspond to a saddle connection, and each horizontal side decomposes into horizontal saddle connections (each one appearing on the top and on the bottom). We can freely change the lengths of these saddle connections hence we can assume they are all of length one, and get a square tiled surface with singularities in each corner of the squares. Now we rotate the vertical one until it makes an angle of $\pi/3$ with the horizontal ones (see Figure~\ref{squares:to:triangles}), this gives the surface $S$ required. 
\begin{figure}
\begin{tikzpicture}[scale=1]
\coordinate (v0) at (1,0);
\coordinate (v1) at (0.5,0.866);
\coordinate (v2) at ($(v0)-(v1)$);
\coordinate (u) at (0,1);
\coordinate (v3) at ($-1*(v0)$);
\coordinate (v4) at ($-1*(v1)$);

\draw (0,0) --++ (v0) node{$\bullet$} --++ (v0) node{$\bullet$} --++ (v0) node{$\bullet$} --++ (v0) node{$\bullet$} --++ (v0) node{$\bullet$} --++ (v0) node{$\bullet$} --++ (v0) node{$\bullet$} --++ (v0) node{$\bullet$} --++ (u) node{$\bullet$} --++ (v3) node{$\bullet$} --++ (v3) node{$\bullet$}--++ (v3) node{$\bullet$}--++ (v3) node{$\bullet$}--++ (v3) node{$\bullet$}--++ (v3) node{$\bullet$}--++ (v3) node{$\bullet$}--++ (v3) node{$\bullet$} --++ (0,-1) node{$\bullet$};

\draw (0,-3) --++ (v0) node{$\bullet$} --++ (v0) node{$\bullet$} --++ (v0) node{$\bullet$} --++ (v0) node{$\bullet$} --++ (v0) node{$\bullet$} --++ (v0) node{$\bullet$} --++ (v0) node{$\bullet$} --++ (v0) node{$\bullet$} --++ (v1) node{$\bullet$} --++ (v3) node{$\bullet$} --++ (v3) node{$\bullet$}--++ (v3) node{$\bullet$}--++ (v3) node{$\bullet$}--++ (v3) node{$\bullet$}--++ (v3) node{$\bullet$}--++ (v3) node{$\bullet$}--++ (v3) node{$\bullet$} --++ (v4) node{$\bullet$};

\draw ($(0,-3)+(v1)$)--++(v2)--++(v1)--++(v2)--++(v1)--++(v2)--++(v1)--++(v2)--++(v1)--++(v2)--++(v1)--++(v2)--++(v1)--++(v2)--++(v1)--++(v2);

\draw[very thick,->] (4,-0.5)--++(0,-1);
\end{tikzpicture}
\caption{Surface with a equilateral triangle decomposition}
\label{squares:to:triangles}
\end{figure}

The proof of (2) is a small variation of the above proof: observe first that each singularity appears both on the top line and on the bottom line of the cylinder. Recall that $\SL(2,\mathbb{R})$ acts on the connected component of the stratum by linear action on the polygons. Then applying the matrix $\left(\begin{smallmatrix} 1&n \\ 0 & 1 \end{smallmatrix} \right)$ and suitably cutting and pasting we obtain a new rectangle. For a suitable $n$ there is a vertical length one saddle connection joining the singularity of degree $k_i$ to the singularity of degree $k_j$, and the above argument finishes the proof.

\end{proof}

\begin{theorem}\label{th:max:sys}
Let $S$ be a genus $g\geq 1$ translation surface of area one and $r>0$ singularities or marked points. Then
$$\mathrm{Sys}(S) \leq \left(\frac{\sqrt{3}}{2}(2g-2+r)\right)^{-\frac{1}{2}}.$$
The equality is obtained if and only if $S$ is built with equilateral triangles with sides saddle connections of length $\mathrm{Sys}(S)$. Such a surface exists in any connected component of any stratum.
\end{theorem}
\begin{proof}
For simplicity, instead of looking at a translation surface of area one and trying to determine the longest systole possible, we suppose that $S$ has a systole of length $1$ and we try to minimize the area $\mathcal{A}(S)$.

We consider a Delaunay triangulation of $S$ given by saddle connections. By Lemma~\ref{L:ShortSaddleDelaunay} all shortest saddle connections of $S$ are 1-cells in this triangulation. Note that some triangles in the Delaunay triangulation might have small area.

 We consider the Voronoi diagram of $S$. This is a partitioning of $S$ into cells. Each cell contains exactly one singularity and is the set of points of $S$ that are closer to that singularity than to any other. The boundary of each cell consists of points that are equidistant to at least two singularities in the sense that there are at least two different distance realizing geodesics of equal length connecting the point with a singularity  (see Section~4 of \cite{MS91} for reference). 

 The boundaries of the cells of the Voronoi diagram are parts of the orthogonal bisectors of the saddle connection in the Delaunay triangulation. Even though the triangulation is not unique, the Voronoi diagram is unique.

 We can compute $\mathcal{A}(S)$ as the sum of the areas of the triangles with one of the vertices a singularity and its opposite side a side of the Voronoi cell containing the singularity (see Figure~\ref{fig:Del:Vor}). The height of such a triangle is a half of a saddle connection and hence its length is greater than or equal to $\frac 1 2$. Therefore $\mathcal{A}(S)$ is greater or equal to one half of the sum of the lengths of all the sides of the Voronoi cells.
 
 \begin{figure}[htb]
\begin{tikzpicture}[scale=1.5]
\fill[gray!25] (0,0)--(1,-1) -- (1,0.5) --cycle;

\draw (0,0) node {$\bullet$} node[above] {$A$};
\draw (-1.5,-0.5)--(1,-1) -- (1,0.5) -- (-0.5,2)-- (-2,0.5) ;
\draw (0,0)--(2,0) node {$\bullet$} node[above] {$B$};
\draw (0,0)--++ ($1.5*(1,1)$) node {$\bullet$} node[right] {$C$} -- (2,0);
\draw[dashed] (-2,0.5)--(-1.5,-0.5);
\draw (1.1,0)--(1.1,0.1)--(1,0.1);
\draw (0.82,0.82)--(0.75,0.89)--(0.68,0.82);
\draw (-1.3,0.5) node {$\mathcal{V}(A)$};
\end{tikzpicture}
\caption{A Delaunay triangle $T=\triangle ABC$, with the Voronoi cell $\mathcal{V}(A)$ containing $A$. The area of the gray triangle is at least one half times   the length of corresponding side of the Voronoi cell.}
\label{fig:Del:Vor}
\end{figure}

For each triangle $T$ in the Delaunay triangulation we consider the sum $\sigma(T)$ of the signed distances from the circumcenter of $T$ to its sides. The sum of the lengths of all the sides of the Voronoi cells equals the sum of $\sigma(T)$ of all $T$ in the triangulation. We want to bound from below $\sigma(T)$ for each triangle $T$.

By Carnot's theorem\footnote{Lazare Carnot 1753-1823.} $\sigma(T)$ is equal to the sum of the inradius and the circumradius of $T$ (see for instance \cite{ref:Carnot}). Hence by Lemma~\ref{lem:Rr}, $\sigma(T)\geq \frac{\sqrt 3}{2}$ with equality if and only if $T$ is equilateral of side 1.

The number of triangles in the triangulation is $2(2g-2+r)$. Hence $\mathcal{A}(S) \geq \frac{\sqrt{3}}{2}(2g-2+r)$ if the systole is of length $1$. Thus for a translation surface of area one, we have that the systole is at most $\left(\frac{\sqrt{3}}{2}(2g-2+r)\right)^{-\frac{1}{2}}$ and can be obtained only if $S$ is built with equilateral triangles with sides saddle connections of length $\mathrm{Sys}(S)$.

We conclude by using the first statement of Lemma~\ref{lemma:decomp:triang:equil}.
\end{proof}

\begin{lemma}\label{lem:Rr}
Let $T$ be a nondegenerate Euclidean triangle with sides of length at least 1.  Then, the sum of the circumradius and the inradius of $T$ is at least $\frac{\sqrt 3}{2}$, with equality if and only if $T$ is equilateral of side~1.
\end{lemma}
\begin{proof}
Denote by $R$  the circumradius and by $r$ the inradius of $T$.
First we note that when we shrink $T$ we decrease the sum $R+r$. So without loss of generality, we can assume that at least one of the sides of $T$ is of length $1$. So for the triangle $T=\triangle ABC$ with $1=AB \leq BC \leq AC$ we take a point $D$ on the side $BC$ so that $BD=AB$. Note that $AD\geq 1$.  For the  inradius $\tilde{r}$ and the circumradius $\tilde{R}$ of the isosceles $\triangle ABD$ we can see that $\tilde{r}\leq r$ and $\tilde{R}\leq R$. Indeed, the circumcenter of $\triangle ABD$ is nearer to $AB$ than the circumcenter of $\triangle ABC$ and therefore $\tilde{R}\leq R$.  And to obtain that $\tilde{r}\leq r$, we note that the incenter of $\triangle ABD$ is nearer to $B$ than the incenter of $\triangle ABC$.

For a triangle with sides $1,1$ and $x$, we can find the inradius and the circumraduis with the help of the lengths of the sides:
$$ \tilde{R}(x) = \sqrt \frac{1}{4 - x^2}, \qquad \tilde{r}(x) = \frac{x}{2} \sqrt \frac{2-x}{2+x},$$
with $x \in [1,2)$. For the sum $(\tilde{R}+\tilde{r})(x)$ and its derivative we obtain
$$(\tilde{R}+\tilde{r})(x) = \frac{2+2x-x^2}{2\sqrt{4-x^2}}, \quad (\tilde{R}+\tilde{r})'(x) = \frac{8-6x+x^3}{2\sqrt{(4-x^2)^3}.}$$
Since $8-6x+x^3=x(1-x)^2+2(2-x)^2+x > 0$ for $x\in [1,2)$, we have that $(\tilde{R}+\tilde{r})(x)$ is strictly increasing in the interval $[1,2)$ and hence obtains its minimum for $x=1$. Therefore $R+r \geq \tilde{R}+\tilde{r} \geq \frac{\sqrt{3}}{2}$ with equality exactly when the triangle $T$  is equilateral with side 1.

\end{proof}

\section{Locally maximal systole}
The question is if there exists local but not global maxima in any given stratum $\mathcal{H}_1(k_1,\ldots,k_r)$ of translation surfaces of area one. Note that such maxima is never strict since rotating a translation surface preserves the systole. We denote by $\mathbb{P}\mathcal{H}(k_1,\dots ,k_r)$ the moduli space of translation surfaces in $\mathcal{H}(k_1,\dots ,k_r)$ up to rotation and scaling. The systole function is well defined in $\mathbb{P}\mathcal{H}(k_1,\dots ,k_r)$:  for $[S]\in \mathbb{P}\mathcal{H}(k_1,\dots ,k_r)$, we define $\mathrm{Sys}([S])$ to be $\mathrm{Sys}(S)$, where $S$ is any area one representative of $[S]$.

In this section, we show examples of local maxima of the function $\mathrm{Sys}$ that are not global and prove that such examples are realized in all but a finite number of strata.

We need first, for technical reasons, to define a distance around a point in $\mathcal{H}(k_1,\ldots,k_r)$ and in $\mathbb{P}\mathcal{H}(k_1,\ldots,k_r)$. Let $S_0\in \mathcal{H}(k_1,\ldots,k_r) $. Fix a basis of the relative homology given by saddle connections that determines local coordinates $(v_1,\dots ,v_k)$ around $S_0$. Then for $S$ in a sufficiently small neighborhood of $S_0$, we define $d(S,S_0)=\max_i\{|v_i-v_{i_0}|\}$.

We will identify a sufficiently small neighborhood of an element $[S_0]\in \mathbb{P}\mathcal{H}(k_1,\dots ,k_r)$, with the subset of representatives in $\mathcal{H}(k_1,\dots ,k_r)$ normalized in the following way:
\begin{enumerate}
\item the first coordinate $v_1$ is in $]0,+\infty [$,
\item the length of the shortest saddle connection is 1.
\end{enumerate}
Then, the distance to $[S_0]$ is the distance in $\mathcal{H}(k_1,\ldots,k_r)$ following this identification.

\begin{theorem} \label{th:local:maxima}
Let $S_{reg}$ be a translation surface in $\mathcal{H}_1(k_1,\ldots,k_r)$ such that when cut along its saddle connections of length $\mathrm{Sys}(S_{reg})$, it decomposes to equilateral triangles and regular hexagons so that:
\begin{itemize}
\item  the set of the equilateral triangles without the vertices is connected, 
 \item  the boundary of each polygon is contained in the boundary of the set of triangles. 
 \end{itemize}
Then $\mathrm{Sys}(S_{reg})$ is a local maximum in $\mathcal{H}_1(k_1,\ldots,k_r)$ and even a strict local maximum in $\mathbb{P}\mathcal{H}(k_1,\ldots,k_r)$.
\end{theorem}

\begin{rem}
The second condition of  the above statement is equivalent to having the hexagons neither adjacent nor self-adjacent.
\end{rem}

The idea of the proof is the following: when deforming a little $[S_{reg}]$ following the normalization described above, the area of each triangle does not decrease, the area of each hexagon might decrease, but this will be compensated by an increase coming from at least one triangle. 


The next two lemmas are estimations of the variation of areas of hexagons and triangles that are deformed in our context.
\begin{lemma} \label{L:perturbation:hexagon}
Let $H_{reg}$ be the regular hexagon of sides of length $1$.
There exists a positive constant $c$ such that for every $\varepsilon>0$   small enough and every  convex hexagon $H=A_1 A_2 \ldots A_6$ with sides of lengths in the interval $[1, 1 + \varepsilon]$ and diagonals $A_1A_3$, $A_3A_5$ and $A_5A_1$ of lengths in the interval $[\sqrt{3}-\varepsilon,\sqrt{3}+\varepsilon]$, we have $\mathrm{Area}(H) \geq \mathrm{Area}(H_{reg}) - c\varepsilon^2$.
\end{lemma}
\begin{proof}
We consider the convex hexagon $H'=A_1A_2'A_3A_4'A_5A_6'$ such that all of its sides are of length 1 (see Figure~\ref{Fig:hex} ). We see that $\mathrm{Area}(H) \geq \mathrm{Area}(H')$.

\begin{figure}
\begin{tikzpicture}[scale=2]
\coordinate (v0) at (1,0);
\coordinate (v1) at (0.5,0.866);
\coordinate (v2) at ($(v1)-(v0)$);
\coordinate (v3) at ($-1*(v0)$);
\coordinate (v4) at ($-1*(v1)$);
\coordinate (v5) at ($(v0)-(v1)$);
\coordinate (h1) at (0.1,-0.1);
\coordinate (h2) at (0.2,0.1);
\coordinate (h3) at (-0.2,-0.1);

\draw[dashed] (0,0)--++ (v0)--++ (v1)--++ (v2) --++ (v3) --++ (v4)-- cycle;
\draw (0,0)--++ ($(v1)+(v0)$) node[midway, above] {$d_1$}--++ ($(v3)+(v2)$) node[midway, below] {$d_2$}-- cycle node[midway, right] {$d_3$}; 
\draw (0,0)--++ ($(v0)+(h1)$)--++ ($(v1)-(h1)$)--++ ($(v2)+(h2)$) --++ ($(v3)-(h2)$) --++ ($(v4)+(h3)$)-- cycle;
\end{tikzpicture}
\caption{The hexagon $H$ and the new hexagon $H'$ of side 1  (dashed).}
\label{Fig:hex}
\end{figure}

We note the lengths of the diagonals $A_1A_3$, $A_3A_5$ and $A_5A_1$ by $d_1$, $d_2$ and $d_3$ respectively. The area of the hexagon $H'$ depends smoothly on $(d_1,d_2,d_3)$ and admits a local minimum at the point $(\sqrt{3},\sqrt{3},\sqrt{3})$ (that corresponds to the regular hexagon).

Therefore by the Taylor-Young formula we obtain
$$ Area(H) = Area(H_{reg}) + o(||(d_1-\sqrt{3},d_2-\sqrt{3},d_3-\sqrt{3})||^2). $$
Since for $i \in \{1,2,3\}$ we have $d_i \in [\sqrt{3}-\varepsilon,\sqrt{3}+\varepsilon]$ there exists a constant $c \in \mathbb{R}$ such that 
$$\mathrm{Area}(H) \geq \mathrm{Area}(H_{reg}) - c\varepsilon^2.$$

\end{proof}

\begin{lemma} \label{L:perturbation:triangle}
Let $T_{reg}$ be the equilateral triangle of sides of length 1.
There exists a positive constant $c\in \R$ such that  for every $\varepsilon>0$   small enough and every triangle $T$ with one of its sides of length $1+\varepsilon$ and the other sides of lengths in the interval $[1,1+\varepsilon]$, we have that $\mathrm{Area}(T) > \mathrm{Area}(T_{reg}) + c\varepsilon$.
\end{lemma}
\begin{proof}
Let $T=\triangle ABC$ and $d(A,B)=1+\varepsilon$ and let $C'$ be such that $d(A,C')=d(B,C')=1$. We have $\mathrm{Area}(\triangle ABC) \geq \mathrm{Area}(\triangle ABC')$. 

By Heron's formula, the area of $\triangle ABC'$ is:
$$
\mathrm{Area}(\triangle ABC')  = \frac{1}{4}\sqrt{(3+\varepsilon)(1-\varepsilon)(1+\varepsilon)^2}
=\frac{\sqrt 3}{4}+\frac{\sqrt 3}{6}\varepsilon + o(\varepsilon).$$
Therefore there exists a constant $c>0$ such that for all $\varepsilon$ small enough we have $\mathrm{Area}(T) > \mathrm{Area}(T_{reg}) + c\varepsilon$.
\end{proof}

\begin{lemma}\label{L:perturbation:longueur:triangle}
Let $ABC$ be a nondegenerate triangle of sides of length $l_1=BC$, $l_2=AC$, and $l_3=AB$. For $\varepsilon$ small enough, let $A'B'C'$ be a triangle with sides of lengths $l_1',l_2',l_3'$  such that for each $i\in \{1,2,3\}$, $|l_i-l_i'|\leq\varepsilon$. We assume further that  $d(A,A') \leq \varepsilon$,  $d(B,B') \leq \varepsilon$ and $C$ and $C'$ are in the same half-plane determined by $AB$. Then there is a constant $J>1$, only depending on $l_1,l_2,l_3$ such that $d(C,C') \leq  J\varepsilon$.
\end{lemma}

\begin{proof}
We consider first the translation $\tau$ of $\R^2$ of direction $\overrightarrow{A'A}$. We remark that $\tau(A')=A$ and $d(B',\tau(B'))<2\varepsilon$. Then we consider the rotation $\rho$ with center $A$ and of angle $\measuredangle BA\tau(B')$. We note $X''=\rho(\tau(X'))$ where $X \in \{A',B',C'\}$. See Figure~\ref{fig:lem:4.5}. We remark that $A$, $B$ and $B''$ are on the same line and that
$$d(\tau(C'),C'')=\frac{d(\tau(A'), \tau(C'))}{d(\tau(A'), \tau(B'))}d(\tau(B'),B'') .$$
Since $d(\tau(B'),B'') \leq d(\tau(B'),B)+d(B,B'') < 2\varepsilon + \varepsilon$, we obtain for $\varepsilon$ small enough a constant $J_1=J_1(l_2,l_3)$ such that
$$ d(\tau(C'),C'') < J_1\varepsilon. $$

\begin{figure}[htb]
\begin{tikzpicture} [scale=3.5]
\coordinate (v0) at (1,0);
\coordinate (vv0) at (1.05,0.05);
\coordinate (v1) at (0.5,0.866);
\coordinate (vv1) at (0.6,0.966);
\coordinate (e) at (-0.05,-0.05);
\coordinate (B) at (2,0);

\draw (0,0) node[above right] {$A$}-- (v0) node[below] {$B$} -- (v1) node[above] {$C$} -- cycle;
\draw[dashed] (e)  node[below left] {$A'$} --(vv0)  node[right] {$B'$} -- (vv1)  node[right] {$C'$}-- cycle;
\draw (B) node[below] {$A=A''$}-- ($(v0)+(B)$) node[below] {$B$} -- ($(v1)+(B)$) node[above] {$C$} -- cycle;
\draw[->] (1.2,0.4) --++ (0.6,0);

\begin{scope} [shift={(2.05cm,0.05cm)},rotate=-5.3]
\draw[dashed] (-0.05,-0.05) --(1.05,0.05) node[below right] {$B''$}  -- (0.6,0.966) node[right] {$C''$}-- cycle;
\end{scope}
\end{tikzpicture}
\caption{The triangles $ABC$, $A'B'C'$ and $A''B''C''$.}
\label{fig:lem:4.5}
\end{figure}

We want to bound $d(C,C'')$. 
For $(M, t)$ in a neighborhood of $(C,l_3)$, we consider the triangle $AMN_t$ where $N_t$ is in the ray $AB$ with $d(A,N_t)=t$ 
 and we define $\phi(M,t)=(d(M,A), d(M, N_t),d(A,N_t))$. The map $\phi$ is  smooth and its Jacobian derivative at $(C,l_3)$ is invertible. Hence, it defines a locally invertible map and $\phi^{-1}$ is smooth. This implies that there is a constant $J_2=J_2(l_1,l_2,l_3)$ such that for $\varepsilon$ small enough $$d(C,C'')<J_2 \varepsilon.$$

Combining with the above estimations, we obtain $d(C,C')<(J_1+J_2+1)\varepsilon$.
\end{proof}

\begin{proof}[Proof of Theorem~\ref{th:local:maxima}]
We show directly that $\mathrm{Sys}([S_{reg}])$ is a strict local maximum in the projective stratum, and replace $S_{reg}$ by a surface, still denoted $S_{reg}$ with shortest saddle connections of length one.

First, we remark that removing all shortest saddle connections of $S_{reg}$ gives a union of topological disks.  Hence we can find a basis of the relative homology that consists of shortest saddle connections $(\gamma_1,\dots ,\gamma_k)$ and we can assume that $\gamma_1$ is horizontal and oriented from left to right. We use this basis to fix local coordinates of the stratum $\mathcal{H}(k_1,\dots ,k_r)$, and define a distance in a neighborhood of $S_{reg}$. 
Recall that we identify a neighborhood of  element $[S_{reg}]\in \mathbb{P}\mathcal{H}(k_1,\ldots,k_r)$ with a subset $\mathcal{U}$ of $\mathcal{H}(k_1,\dots ,k_r)$ satisfying the following conditions: the shortest saddle connection is of length 1 and $\gamma_1$ stays horizontal. For $S\in \mathcal{U}$, we call \emph{short saddle connection} any saddle connection that corresponds to a shortest saddle connection of $S_{reg}$.

Let $\varepsilon>0$ be small enough and $S\in \mathcal{U}$ be such that  $\varepsilon=d(S,S_{reg})$. Let us define $\rho(S)=Max_{\gamma}(l(\gamma)-1)$, where the maximum is taken on all short saddle connections of $S$. By hypothesis, $\rho(S)\geq 0$. 

In a more general setting, we prove in Section~\ref{sec:rigid} that we have $\rho(S)=0$ if and only if $S=S_{reg}$. However in the current proof we need a stronger result (see the claim below).

We observe that since any short saddle connection $\gamma$ is a linear combination of $\{\gamma_1,\dots ,\gamma_k\}$ in the relative homology group, then its corresponding affine holonomy $v_\gamma$ satisfies $|v_{\gamma}-v_{\gamma,reg}|\leq K\varepsilon$. Since there are only a finite number of short saddle connections, $K$ can be made universal for  all short saddle connections. In particular, $\rho(S)\leq K \varepsilon$.


We have the following facts:
\begin{enumerate}
\item  
The sides of each hexagon $H$ in $S$ corresponding to a regular hexagon $H_{reg}$ in the decomposition of $S_{reg}$ are short saddle connections. By the above observation, we can apply Lemma~\ref{L:perturbation:hexagon} to $H$ for $\varepsilon'=2K\varepsilon$. Hence, there is a constant $c_1$, such that $$Area(H)\geq Area(H_{reg})-c_1\varepsilon^2.$$

\item By Lemma~\ref{L:perturbation:triangle}, there exists at least one equilateral triangle $T_{reg}$ in the decomposition of $S_{reg}$, such that for  the corresponding triangle $T$ in $S$ we have $Area(T)\geq Area(T_{reg})+c_2\rho(S)$ where $c_2$ is a positive constant. 
Furthermore the area of each triangle in $S_{reg}$ is not greater than the area of the corresponding triangle in $S$.
Summing up the corresponding contributions of the triangles, we obtain
$$Area(\cup T)\geq Area(\cup T_{reg})+c_2\rho(S).$$
\end{enumerate}

{\bf Claim:} There is a constant $D$ such that for $\varepsilon=d(S,S_{reg})$ small enough, $\varepsilon <D\rho(S)$. 
In other words: lengths of short saddle connections control the distance from $S$ to $S_{reg}$.

Summing up all contributions, assuming the claim, we see that the  area of $S$ is greater than the area of $S_{reg}$ for $\varepsilon>0$ small enough. Hence $S_{reg}$ is a local maximum of $\mathrm{Sys}$ which is nonglobal since the surface $S_{reg}$ is not built with equilateral triangles of sides saddle connections.
\medskip

Now we prove the claim. 
Recall that we assume that $\gamma_1$ does not change direction. Let $\delta=\rho(S)$.

Let $\gamma\in \{\gamma_2,\dots ,\gamma_k\}$ be a saddle connection in the fixed basis. By hypothesis, there is a sequence of pairwise distinct equilateral triangles $T_1,\dots ,T_l$ (whose sides are length one saddle connections) that form a ``path'' from $\gamma_1$ to $\gamma$, \emph{i.e.} such that
\begin{enumerate}
\item $\gamma_1$ is a side of $T_1$,
\item for each $i\in \{1,\dots ,l-1\}$, $T_i$ and $T_{i+1}$ are adjacent,
\item $\gamma$ is a side of $T_l$.
\end{enumerate}
Observe that $l$ is bounded from above by the total number $N$ of triangles in the decomposition of $S_{reg}$. Denote by $v_{reg}$ the affine holonomy of $\gamma$ in $S_{reg}$ and by $v$ the affine holonomy of $\gamma$ in $S$. We will use Lemma~\ref{L:perturbation:longueur:triangle} to bound $|v-v_{reg}|$.

Using the developing map (see Figure~\ref{seq:adj:triangles}), we can view the triangles $(T_i)_i$ as a sequence of adjacent equilateral triangles of the plane although in this case the triangles might intersect. We deform the surface $S_{reg}$ to obtain the surface $S$. The triangles $(T_i)_i$ persist but are not necessarily equilateral any more. Again, we can view them as a sequence of adjacent triangles $(T_i')_i$ in the plane. 

\begin{figure}[htb]
\begin{tikzpicture}[scale=3]
\coordinate (v0) at (1,0);
\coordinate (v1) at (0.5,0.866);
\coordinate (v2) at ($(v1)-(v0)$);
\coordinate (v3) at ($-1*(v0)$);
\coordinate (v4) at ($-1*(v1)$);
\coordinate (v5) at ($(v0)-(v1)$);

\draw (0,0) node[above] {$A_1=A_1'$} --++ (v0) node[above left] {$B_1$} --++ (v5) --++ (v1) --++ (v5) --++ (v4) --++ (v3) --++ (v2) --++ (v2);
\draw (v0) --++ (v4) --++ (v0) --++ (v4);
\draw ($2*(v5)+(v0)$)--++ (v2) --++ (v0);
\draw[dashed] (0,0)--++ ($1.1*(v0)$) node[above right] {$B_1'$} --++ ($1.05*(v5)-0.05*(v0)$)--++ ($(v1)+0.1*(v0)$);
\draw[dashed] (0,0) --++ ($(v5)+0.1*(v4)$)--($1.1*(v0)$);
\draw[dashed] ($(v5)+0.1*(v4)$) -- ($1.1*(v0)+1.05*(v5)-0.05*(v0)$) -- ($2*(v5)+0.1*(v4)$)-- ($(v5)+0.1*(v4)$);
\draw[dashed] ($2*(v5)+0.1*(v4)$) --++ ($(v0)+0.1*(v5)$)-- ($1.1*(v0)+1.05*(v5)-0.05*(v0)$)--++($(v0)+0.1*(v5)$)-- ($(v0)+0.1*(v5)+2*(v5)+0.1*(v4)$);
\draw[dashed,very thick] [->] ($1.1*(v0)+1.05*(v5)-0.05*(v0)+(v1)+0.1*(v0)$)--($1.1*(v0)+1.05*(v5)-0.05*(v0)+(v0)+0.1*(v5)$) node[midway,right] {$v$};
\draw[very thick] [->] ($(v0)+(v5)+(v1)$) --++ (v5) node[midway,below] {$v_{reg}\ \ $};

\draw (0.5,-0.3) node {$T_1$};
\draw (1,-0.6) node {$T_2$};
\end{tikzpicture}

\caption{A sequence of adjacent triangles and the perturbed ones}
\label{seq:adj:triangles}

\end{figure}

Denote by $T_1=A_1B_1C_1$ and $T_1'=A_1'B_1'C_1'$. We can assume that $A_1=A_1'$ is the vertex neither in $T_2$ nor in $T_2'$, and $B_2,B_2'$ are such that the segments $A_1B_1$ and $A_1'B_1'$ are horizontal (see Figure~\ref{seq:adj:triangles}). More generally for $i>1$, denote the triangle $T_i$ by $A_iB_iC_i$ in such a way that $A_iB_i$ is a side of previous triangle and  that $B_iC_i$ is a side of the next triangle, and we denote analogously the vertices of $T_i'$. Using Lemma~\ref{L:perturbation:longueur:triangle} we see that $d(C_1,C_1')<J\delta$ (recall that since $\rho(S)<Kd(S,S_{reg})=K\varepsilon$, we can assume $\delta$ to be arbitrarily small).  Since $d(B_1,B_1')<\delta<J\delta$ we can apply Lemma~\ref{L:perturbation:longueur:triangle} to the triangles $T_2$ and $T_2'$ for the constant $J\delta$ and we get $d(C_2,C_2')<J^2 \delta$. Since $l$ is bounded from above by $N$ and $\delta$ can be chosen arbitrarily small, we get $d(C_l,C_l')<J^l \delta$ and $d(B_l,B_l')<J^{l-1}\delta$. Finally, observe that $v$ is given by the difference of the coordinates of $B_l'$ and $C_l'$, and therefore:
$$|v-v_{reg}|< (J^l+J^{l-1})\delta < 2.J^N \delta. $$
This concludes the proof of the claim and of the theorem.

\end{proof}

\begin{example}\label{ex:loc:non:glob}
The surfaces given in Figure~\ref{fig:examples} are examples (with one hexagon) of local maxima that are nonglobal  in the strata $\mathcal{H}(2,0^k)$ and $\mathcal{H}(1,1,0^k)$, for $k\geq 1$.
\end{example}

\begin{figure}[htb]

\begin{tikzpicture}[scale=1.1]
\coordinate (v0) at (1,0);
\coordinate (v1) at (0.5,0.866);
\coordinate (A) at (-0.5,0);
\coordinate (B) at (-0.5,-3);
\coordinate (C) at (0,-6);
\coordinate (D) at (0,-9);
\coordinate (E) at (2,-10);

\draw  (A) --++ (v0) node[midway,below] {$n$} node {\tiny $\bullet$} --++ (v1)  node {\tiny $\bullet$} --++ ($(v1)-(v0)$) node {\tiny $\bullet$} --++ ($-1*(v0)$) node[midway,above] {$1$} node {\tiny $\bullet$}--++ ($-1*(v1)$) node[midway,above left] {$a$} node {\tiny $\bullet$}--++ ($(v0)-(v1)$)  node[midway,below left] {$b$} node {\tiny $\bullet$};
\draw ($(A)+(v0)$) --++ (v0)  node[midway,below] {$1$} node {\tiny $\bullet$};
\draw[dashed] ($(A)+2*(v0)$) --++ ($2*(v0)$) node {\tiny $\bullet$};
\draw ($(A)+4*(v0)$) --++ (v0) node[midway,below] {$n-1$} node {\tiny $\bullet$}--++ (v1) node[midway,below right] {$a$} node {\tiny $\bullet$} --++  ($(v1)-(v0)$) node[midway,above right] {$b$} node {\tiny $\bullet$} --++ ($-1*(v0)$) node[midway,above] {$n$} node {\tiny $\bullet$} --++ ($-1*(v0)$) node[midway,above] {$n-1$} node {\tiny $\bullet$};
\draw[dashed] ($(A)+2*(v1)$) --++ ($2*(v0)$);
\draw[dashed] ($(A)+3*(v0)+(v1)$) node {\tiny $\bullet$};
\draw ($(A)+4*(v0)+(v1)$) node {\tiny $\bullet$};
\draw ($(A)+8.2*(v0)+(v1)$)  node {\small in $\mathcal{H}(2,0^{2n-3})$};
\draw ($(A)+8.2*(v0)+(v1)+(0,-0.5)$)  node {\small for $n\geq 2$};

\draw  (B) --++ (v0) node[midway,below] {$n$} node {\tiny $\bullet$} --++ (v1)  node {\tiny $\bullet$} --++ ($(v1)-(v0)$) node {\tiny $\bullet$} --++ ($-1*(v0)$) node[midway,above] {$1$} node {\tiny $\bullet$}--++ ($-1*(v1)$) node[midway,above left] {$a$} node {\tiny $\bullet$}--++ ($(v0)-(v1)$) node[midway,below left] {$b$} node {\tiny $\bullet$};
\draw ($(B)+(v0)$) --++ (v0)  node[midway,below] {$2$} node {\tiny $\bullet$};
\draw ($(B)+4*(v0)$) --++ (v0) node[midway,below] {$n-1$} node {\tiny $\bullet$}--++ (v0) node[midway,below] {$1$} node {\tiny $\bullet$} --++ (v1) node[midway,below right] {$a$} node {\tiny $\bullet$} --++  ($(v1)-(v0)$) node[midway,above right] {$b$} node {\tiny $\bullet$} --++ ($-1*(v0)$) node[midway,above] {$n$} node {\tiny $\bullet$} --++ ($-1*(v0)$) node[midway,above] {$n-1$} node {\tiny $\bullet$};
\draw ($(B)+2*(v1)$) --++ (v0) node[midway,above] {$2$} node {\tiny $\bullet$};
\draw[dashed] ($(B)+2*(v0)$) --++ (v0) --++ (v0)  node {\tiny $\bullet$};
\draw[dashed] ($(B)+2*(v1)+(v0)$) --++ ($(v0)$) --++ (v0) ;
\draw ($(B)+4*(v0)+(v1)$) node {\tiny $\bullet$};
\draw ($(B)+5*(v0)+(v1)$) node {\tiny $\bullet$};
\draw ($(B)+8.2*(v0)+(v1)$)  node {\small in $\mathcal{H}(1,1,0^{2n-4})$};
\draw ($(B)+8.2*(v0)+(v1)+(0,-0.5)$)  node {\small for $n\geq 3$};

\draw  (C) --++ (v0) node[midway,below] {$n$} node {\tiny $\bullet$} --++ (v1)  node {\tiny $\bullet$} --++ ($(v1)-(v0)$) node {\tiny $\bullet$} --++ ($-1*(v0)$) node[midway,above] {$2$} node {\tiny $\bullet$}--++ ($-1*(v1)$) node {\tiny $\bullet$}--++ ($(v0)-(v1)$) node[midway,below left] {$b$} node {\tiny $\bullet$};
\draw ($(C)+(v0)$) --++ (v0)  node[midway,below] {$2$} node {\tiny $\bullet$};
\draw ($(C)+4*(v0)$) --++ (v0) node[midway,below] {$n-1$} node {\tiny $\bullet$}--++ (v0) node[midway,below] {$1$} node {\tiny $\bullet$} --++ 
($(v1)-(v0)$) node[midway,above right] {$b$} node {\tiny $\bullet$} --++  
($(v1)-(v0)$) node[midway,above right] {$a$} node {\tiny $\bullet$} --++ ($-1*(v0)$) node[midway,above] {$n$} node {\tiny $\bullet$} --++ ($-1*(v0)$) node[midway,above] {$n-1$} node {\tiny $\bullet$};
\draw ($(C)+2*(v1)-(v0)$) --++ ($-1*(v0)$) node[midway,above] {$1$} node {\tiny $\bullet$} --++ ($(v0)-(v1)$) node[midway,below left] {$a$} node {\tiny $\bullet$};
\draw[dashed] ($(C)+2*(v0)$) --++ (v0) --++ (v0)  node {\tiny $\bullet$};
\draw[dashed] ($(C)+2*(v1)$) --++ ($(v0)$) --++ (v0) ;
\draw ($(C)+3*(v0)+(v1)$) node {\tiny $\bullet$};
\draw ($(C)+4*(v0)+(v1)$) node {\tiny $\bullet$};
\draw ($(C)+7.7*(v0)+(v1)$)  node {\small in $\mathcal{H}(2,0^{2n-4})$};
\draw ($(C)+7.7*(v0)+(v1)+(0,-0.5)$)  node {\small for $n\geq 3$};

\draw  (D) --++ (v0) node[midway,below] {$n$} node {\tiny $\bullet$} --++ (v1)  node {\tiny $\bullet$} --++ ($(v1)-(v0)$) node {\tiny $\bullet$} --++ ($-1*(v0)$) node[midway,above] {$2$} node {\tiny $\bullet$}--++ ($-1*(v1)$) node {\tiny $\bullet$}--++ ($(v0)-(v1)$) node[midway,below left] {$b$} node {\tiny $\bullet$};
\draw ($(D)+(v0)$) --++ (v0)  node[midway,below] {$2$} node {\tiny $\bullet$};
\draw ($(D)+4*(v0)$) --++ (v0) node[midway,below] {$n-1$} node {\tiny $\bullet$}--++ (v0) node[midway,below] {$1$} node {\tiny $\bullet$} --++ 
($(v1)-(v0)$) node[midway,above right] {$a$} node {\tiny $\bullet$} --++  
($(v1)-(v0)$) node[midway,above right] {$b$} node {\tiny $\bullet$} --++ ($-1*(v0)$) node[midway,above] {$n$} node {\tiny $\bullet$} --++ ($-1*(v0)$) node[midway,above] {$n-1$} node {\tiny $\bullet$};
\draw ($(D)+2*(v1)-(v0)$) --++ ($-1*(v0)$) node[midway,above] {$1$} node {\tiny $\bullet$} --++ ($(v0)-(v1)$) node[midway,below left] {$a$} node {\tiny $\bullet$};
\draw[dashed] ($(D)+2*(v0)$) --++ (v0) --++ (v0)  node {\tiny $\bullet$};
\draw[dashed] ($(D)+2*(v1)$) --++ ($(v0)$) --++ (v0) ;
\draw ($(D)+3*(v0)+(v1)$) node {\tiny $\bullet$};
\draw ($(D)+4*(v0)+(v1)$) node {\tiny $\bullet$};
\draw ($(D)+7.7*(v0)+(v1)$)  node {\small in $\mathcal{H}(1,1,0^{2n-5})$};
\draw ($(D)+7.7*(v0)+(v1)+(0,-0.5)$)  node {\small for $n\geq 3$};
\end{tikzpicture}

\caption{Examples of local but nonglobal maxima.}
\label{fig:examples}

\end{figure}

The above examples will be used in the next theorem in order to build examples in most strata.

\begin{theorem}\label{th:loc:max:in:each:stratum}
Let $\mathcal{H}$ be a stratum of area one and genus $g\geq 2$ surfaces. We assume that $\mathcal{H}$ is neither $\mathcal{H}(1,1)$ nor $\mathcal{H}(2)$. Then $\mathcal{H}$ contains local maxima of the function $\mathrm{Sys}$ that are not global.
\end{theorem}

We first prove the following lemma.
\begin{lemma}\label{lem:constr:loc}
We consider the stratum $\mathcal{H}=\mathcal{H}(m_1,\dots ,m_r,x,y)$ with $m_1,\dots ,m_r,x,y\geq 0$. We assume that there exists a surface $S_1\in \mathcal{H}$ that satisfies the hypothesis of Theorem~\ref{th:local:maxima} and such that there is a shortest saddle connection $\gamma_1$ joining a singularity of degree $x$ to a distinct singularity of degree $y$. Then
\begin{itemize}
\item[a)]  For any $n_1,\dots ,n_k,p,q\geq 0$ with $p+q+\sum_i n_i$ even, there exists a local but nonglobal maximum of $Sys$ in the stratum $\mathcal{H}(m_1,\dots ,m_r,p+a+1,q+a+1,n_1,\dots ,n_k)$.
\item[b)] For any $n_1,\dots ,n_k,p\geq 0$ with $p+\sum_i n_i$ even, there exists a local but nonglobal maximum of $Sys$ in the stratum $\mathcal{H}(m_1,\dots ,m_r,p+x+y+2,n_1,\dots ,n_k)$.
\end{itemize}
\end{lemma}

\begin{proof}
By Lemma~\ref{lemma:decomp:triang:equil},  there is  a surface $S_2$ that decomposes into equilateral triangles with sides saddle connections in $\mathcal{H}(p,q,n_1,\dots ,n_k)$, and with a shortest saddle connection $\gamma_2$ joining a singularity of degree $p$ to a (distinct) singularity of degree $q$. 

 We can assume $\gamma_1,\gamma_2$ are vertical and of the same length. Now we glue the two surfaces by the following classical surgery: cut the two surfaces along $\gamma_1$ and $\gamma_2$, and glue the left side of $\gamma_1$ with the right side of $\gamma_2$ and the right side of $\gamma_1$ with the right side of $\gamma_2$. We get a surface in $\mathcal{H}(m_1,\dots ,m_r,p+a+1,q+a+1,n_1,\dots ,n_k)$ that satisfies the hypothesis of Theorem~\ref{th:local:maxima} and hence is a local but nonglobal maximum for $Sys$. This proves Case a).
 
 The proof of Case b) is the same by considering a surface $S_2$ in $\mathcal{H}(p,n_1,\dots ,n_k)$ with a shortest saddle connection joining a singularity of degree $p$ to itself.
\end{proof}

\begin{proof}[Proof of Theorem~\ref{th:loc:max:in:each:stratum}]
Recall that examples of local but nonglobal maxima of $Sys$ in the strata $\mathcal{H}(2,0^k)$ and $\mathcal{H}(1,1,0^k)$, for $k\geq 1$ have already been constructed in Example~\ref{ex:loc:non:glob}. 
It remains to constructs examples in all strata of genus at least 3.

We start from the example $S_{0,2}\in \mathcal{H}(2,0)$ given in Example~\ref{ex:loc:non:glob}. There is a saddle connection joining the two singularities.
\begin{itemize}
\item By Case b) of Lemma~\ref{lem:constr:loc}, there is a local maximum in any stratum of the form $\mathcal{H}(p+4,n_1,\dots ,n_k)$ with $p\geq 0$, $k\geq 0$, and $n_1,\dots ,n_k\geq 0$.
\item By Case a) of Lemma~\ref{lem:constr:loc}, there is a local maximum in any stratum of the form $\mathcal{H}(p+3,q+1,n_1,\dots ,n_k)$ with $p,q\geq 0$, $k\geq 0$, and $n_1,\dots ,n_k\geq 0$.
\end{itemize}
There remains to construct examples in strata with singularities of degree at most 2.
Now we consider $S_{2,0,0}\in \mathcal{H}(2,0,0)$ given in Example~\ref{ex:loc:non:glob}. There is a saddle connection joining the two marked points.
\begin{itemize}
\item By Case b) of Lemma~\ref{lem:constr:loc}, there is a local maximum in any stratum of the form $\mathcal{H}(2,2,n_1,\dots ,n_k)$, with $k\geq 0$, and $n_1,\dots ,n_k\geq 0$.
\end{itemize}
Now we consider $S_{1,1,0,0}\in \mathcal{H}(1,1,0,0)$ given in Example~\ref{ex:loc:non:glob}. There is a saddle connection joining the two marked points.
\begin{itemize}
\item By Case b) of Lemma~\ref{lem:constr:loc}, there is a local maximum in any stratum of the form $\mathcal{H}(1,1,2,n_1,\dots ,n_k)$, with $k\geq 0$, and $n_1,\dots ,n_k\geq 0$.
\item  By Case a) of Lemma~\ref{lem:constr:loc}, there is a local maximum in any stratum of the form $\mathcal{H}(1,1,1,1,n_1,\dots ,n_k)$, with $k\geq 0$, and $n_1,\dots ,n_k\geq 0$.
\end{itemize}

Finally, we have produced examples in all strata of genus $g\geq 2$ except $\mathcal{H}(2)$ and $\mathcal{H}(1,1)$.
\end{proof}

\begin{rem}
We remark that we cannot build with these constructions local maxima in $\mathcal{H}(2)$ and in $\mathcal{H}(1,1)$. Indeed, for $\mathcal{H}(2)$ we need one hexagon and two triangles and there is only one possibility that provides a surface in $\mathcal{H}(2)$. But in this case the hexagon is self-adjacent (see next section for a proof that it not a local maximum). For $\mathcal{H}(1,1)$, we need one hexagon and four triangles, and by checking all the possibilities we see that we cannot built the required example.

We prove in a following paper \cite{BG2} that in these strata (and more generally in any hyperelliptic connected components of strata), any local maximum is a global maximum.
\end{rem}

\section{Number of shortest saddle connections}
In this section, we explore the relations between the (locally) maximal values of the function $\mathrm{Sys}$ is the (locally) maximal number of short saddle connections.
\subsection{Maximal number}
In the case of global maxima, the relation is clear as shown in the next proposition.
\begin{proposition} \label{prop:max:number}
The greatest number of shortest saddle connections of a surface in $\mathcal{H}(k_1,\ldots,k_r)$ is equal to $\sum_{i=1}^r 3(k_i+1)$ and this number is realized if and only if  the surface is a global maximum for the function $\mathrm{Sys}$ in $\mathbb{P}\mathcal{H}(k_1,\dots ,k_r)$.
\end{proposition}
\begin{proof}
Let $S$ be a surface in $\mathcal{H}(k_1,\ldots,k_r)$. We consider two shortest saddle connections $\gamma_1$ and $\gamma_2$ in $S$ starting at the same singularity. 

Let us assume that the conical angle between $\gamma_1$ and $\gamma_2$ is less than~$\frac{\pi}{3}$. Then 
\begin{itemize}
\item either the not common ends of $\gamma_1$ and $\gamma_2$ can be connected by a saddle connection and as consequence this saddle connection is shorter than $\gamma_1$ and $\gamma_2$,
\item or there is a saddle connection between $\gamma_1$ and $\gamma_2$ (starting at the same singularity) that is shorter than them.
\end{itemize}
In both cases we have a contradiction and hence the maximal number of shortest saddle connections starting at a singularity of order $k_i$ is $6(k_i+1)$. This gives us that the total number of shortest saddle connections cannot exceed $\sum_{i=1}^r 3(k_i+1)$. 

This number is the number of 1-cells in the Delaunay triangulation. Hence, by Lemma~\ref{L:ShortSaddleDelaunay}, the surface has this number of shortest saddle connections if and only if its Delaunay triangulation is given by equilateral triangles. By Theorem~\ref{th:max:sys} this situation corresponds precisely to global maxima of the function $\mathrm{Sys}$.
\end{proof}

\subsection{Locally maximal number: rigid surfaces} \label{sec:rigid}
For a given translation surface, one would like to find a path joining this surface to a global maximum for the function $\mathrm{Sys}$. Following the above proposition, a greedy algorithm could be to try to increase the number of shortest saddle connections until we reach a surface with the maximal number. Unfortunately, this algorithm does not always work.
 
We call a surface $S$ in $\mathcal{H}(k_1,\ldots,k_r)$ \textit{rigid} if there exists a punctured neighbourhood of $[S]\in \mathbb{P}\mathcal{H}(k_1,\dots ,k_r)$ where all surfaces have a strictly smaller number of shortest saddle connections. As explained above, the global maxima of the systole function are rigid surfaces.

An example of a rigid surface is every surface $S$ that, when cut along its shortest saddle connections, decomposes into equilateral triangles and polygons with no singularities in the interior satisfying the following conditions:
\begin{itemize}
\item  the set of the equilateral triangles without the vertices is connected, 
 \item  the boundary of each polygon is contained in the boundary of the set of triangles. 
 \end{itemize}
 Indeed, when deforming such a surface in a way that the initial shortest saddle connections stay of the same length,  the set  of triangles is isometrically preserved and therefore the set of polygons. In particular, the examples of Theorem~\ref{th:local:maxima} are rigid surfaces.

We give  another family of examples: consider a surface $S$ as above, but instead of having one, it has 2 connected components of triangles. We further assume that there is a polygon $\mathcal{P}$ such that the sum of the affine holonomy of the set of saddle connections of its boundary associated to each component of triangles is nonzero when orienting the saddle connections according to the natural orientation of the $\partial \mathcal{P}$. Indeed as above, when deforming such a surface in a way that the initial shortest saddle connections stay of the same length, then each connected component of triangles is isometrically preserved, and the condition on the  holonomy implies that the boundary~$\mathcal{P}$ is unchanged, which rigidifies the whole surface. If further the polygons are regular hexagons, we can adapt  the proof of Theorem~\ref{th:local:maxima} to show that these are also local but nonglobal maxima.

The examples given in Figure~\ref{fig:non:max:loc} show that it is not sufficient to be decomposed into equilateral triangles and regular hexagons in order to be a local maximum: in this figure, the shortest saddle connections remain of length one and hence the area of the triangles does not change, but the hexagon is deformed and therefore its area decreases. The first example has one connected component of triangles but the hexagon is self-adjacent. The second  one has two connected components of triangles. Note that the example in $\mathcal{H}(0,0,0)$ can be easily modified to give a surface with true singularities (see Remark~\ref{rem:true:singularities}).
\begin{figure}[htb]

\begin{tikzpicture}[scale=1.5]
\coordinate (v0) at (1,0);
\coordinate (v1) at (0.5,0.866);
\coordinate (v2) at ($(v1)-(v0)$);
\coordinate (v3) at ($-1*(v0)$);
\coordinate (v4) at ($-1*(v1)$);
\coordinate (v5) at ($(v0)-(v1)$);
\coordinate (vv1) at (0.318,0.948);
\coordinate (vv3) at (-0.980,-0.199);
\coordinate (vv5) at (0.662,-0.749);
\coordinate (A) at (0,0);
\coordinate (B) at (4,0);
\coordinate (C) at (0,-4);
\coordinate (D) at (4,-4);

\draw ($(A)+(-1.2,2)$) node {$\mathcal{H}(2)$};
\draw  (A) --++ (v0) node {\tiny $\bullet$} --++ (v1) node[midway,right] {3} node {\tiny $\bullet$} --++ (v2) node[midway,right] {4} node {\tiny $\bullet$} --++ (v3) node {\tiny $\bullet$} --++ (v4) node[midway,left] {3} node {\tiny $\bullet$} --++ (v5) node[midway,left] {4} node {\tiny $\bullet$};
\draw (A) --++ (v5) node[midway,left] {$2$}  node {\tiny $\bullet$} --++ (v1) node[midway,right] {$1$};
\draw ($(A)+ 2*(v1)$)  --++ (v2) node[midway,right] {$2$} node {\tiny $\bullet$} --++ (v4) node[midway,left] {$1$};

\draw  (B) --++ (v0) node {\tiny $\bullet$} --++ (vv1) node[midway,right] {3} node {\tiny $\bullet$} --++ (v2) node[midway,right] {4} node {\tiny $\bullet$} --++ (v3) node {\tiny $\bullet$} --++ ($-1*(vv1)$) node[midway,left] {3} node {\tiny $\bullet$} --++ (v5) node[midway,left] {4} node {\tiny $\bullet$};
\draw (B) --++ (v5) node[midway,left] {$2$} node {\tiny $\bullet$}  --++ (v1) node[midway,right] {1};
\draw ($(B)+ (v0)+(vv1)+ (v2)$)  --++ (v2) node[midway,right] {$2$} node {\tiny $\bullet$} --++ (v4) node[midway,left] {$1$};

\draw[->] (2,0.866) --++ (1,0);

\draw ($(C)+(-1.2,2)$) node {$\mathcal{H}(0,0,0)$};
\draw  (C) --++ (v0) node {\tiny $\bullet$} --++ (v1) node[midway,right] {1} node {\tiny $\bullet$} --++ (v2) node[midway,right] {4} node {\tiny $\bullet$} --++ (v3) node {\tiny $\bullet$}--++ (v4) node[midway,left] {3} node {\tiny $\bullet$}--++ (v5) node[midway,left] {2} node {\tiny $\bullet$};
\draw (C) --++ (v5) node[midway,left] {$4$} node {\tiny $\bullet$}  --++ (v1) node[midway,right] {$3$};
\draw ($(C)+ 2*(v1)$)  --++ (v2) node[midway,right] {$2$} node {\tiny $\bullet$} --++ (v4) node[midway,left] {$1$};

\draw  (D) --++ (v0) node {\tiny $\bullet$} --++ (vv1) node[midway,right] {1} node {\tiny $\bullet$} --++ (v2) node[midway,right] {4} node {\tiny $\bullet$} --++ (vv3) node {\tiny $\bullet$}--++ (v4) node[midway,left] {3} node {\tiny $\bullet$}--++ (vv5) node[midway,left] {2} node {\tiny $\bullet$};
\draw (D) --++ ($-1*(v2)$) node[midway,left] {$4$}  node {\tiny $\bullet$} --++ ($-1*(v4)$) node[midway,right] {$3$};
\draw ($(D)+ (v0)+(vv1)+ (v2)$)  --++ ($-1*(vv5)$) node[midway,right] {$2$} node {\tiny $\bullet$} --++ ($-1*(vv1)$) node[midway,left] {$1$};

\draw[->] (2,-3.1) --++ (1,0);
\end{tikzpicture}

\caption{Examples of nonrigid surfaces in $\mathcal{H}(0,0,0)$ and in $\mathcal{H}(2)$.}
\label{fig:non:max:loc}
\end{figure}

More generally, we have the following proposition:
\begin{proposition} \label{prop:loc:block}
Let $S$ be a translation surface such that, when cut along its saddle connections of shortest length, it decomposes into equilateral triangles and  regular hexagons. If the function $\mathrm{Sys}$  admits a local maximum at $[S]\in \mathbb{P}\mathcal{H}(k_1,\dots ,k_r)$, then $S$ is rigid.
\end{proposition}

\begin{proof}
We assume that $S$ is nonrigid, and deform the surface so that we keep all shortest saddle connections of the same length 1. This deformation does not change the metric on each triangle. Therefore, it must change the metric on at least one hexagon, otherwise the metric would be globally unchanged and the transformation would be just a rotation. In particular, the area of the deformed hexagons must strictly decrease, while the area of the triangles (and the unchanged hexagons) remains the same. Hence the area of the surface decreases and thus  $\mathrm{Sys}([S])$ increases.
\end{proof}

An interesting question is if the converse of the above proposition is true. We can also ask if, in general, any local maximum for $\mathrm{Sys}$ comes from a rigid surface. Note that in general, rigid surfaces do not necessarily give local maxima, as shown in the following example.

\begin{proposition}\label{prop:block:nonloc}
The translation surface given by Figure~\ref{example:rigid:not:maximal} is rigid but it is not a local maximum for the function $\mathrm{Sys}$ in $\mathbb{P}\mathcal{H}$ for $n\geq 3$.
\end{proposition}

\begin{figure}[htb]

\begin{tikzpicture}[scale=1.4]
\coordinate (v0) at (1,0);
\coordinate (v1) at (0.5,0.866);
\coordinate (v2) at ($(v1)-(v0)$);
\coordinate (v3) at ($-1*(v0)$);
\coordinate (v4) at ($-1*(v1)$);
\coordinate (v5) at ($(v0)-(v1)$);
\coordinate (vv0) at (1,-0.3);
\coordinate (vvv0) at (1,0.3);
\coordinate (vv1) at ($(v1)+(0,-0.3)$); 
\coordinate (vvv1) at ($(v1)+(0,0.3)$); 
\coordinate (vv2) at ($(v2)+ (0,0.3)$);
\coordinate (vv3) at ($-1*(vv0)$);
\coordinate (vvv3) at ($-1*(vvv0)$);

\coordinate (A) at (0,0);
\coordinate (B) at (0,-5);


\draw  (A) node {$\bullet$} --++ (v0) node {$\bullet$} node[midway, above] {$a$} --++ (v0) node {$\bullet$} --++ (v0) node {$\bullet$}--++ (v2) node {$\bullet$} --++ (v4) --++ (v2) node {$\bullet$} --++ (v4) node[midway,left] {$e$} --++ (v4) node {$\bullet$} --++ (v2) node[midway,left] {$b$}; 
\draw ($(A) + (v5)$) --++ (v4) node {$\bullet$} node[midway,left] {$c$}  --++ (v0)  node {$\bullet$} node[midway,below] {$a$} --++ (v0) node {$\bullet$}  node[midway,below] {$n$} --++ (v0) node {$\bullet$} node[midway,below] {$n-1$};
\draw ($(A) + (v5)$) --++ (v5);
\draw ($(A)+ (v1)+ (v0)$) --++ (v0) node[midway, above] {$1$} --++ (v0) node {$\bullet$} node[midway, above] {$2$} --++ (v4);
\draw[dashed] ($(A)+ 3*(v0)$) --++ ($3*(v0)$);
\draw[dashed] ($(A)+ 3*(v0)+ (v5)+(v4)$) --++ ($3*(v0)$);
\draw[dashed] ($(A)+4*(v0)$)--++ (v2) --++ (v0) --++ (v0) --++ (v4) --++ (v2)--++ (v4);
\draw ($(A)+6*(v0)$) node {$\bullet$} --++ (v0)  node {$\bullet$} --++ (v0)  node {$\bullet$} --++ (v0)  node {$\bullet$} node[midway,above] {$d$} --++ (v4)  node {$\bullet$}  node[midway,right] {$e$}--++ (v5)  node {$\bullet$}  node[midway,right] {$b$}--++ (v3)  node {$\bullet$} node[midway,below] {$d$} --++ (v3)  node {$\bullet$} node[midway,below] {$1$}--++ (v3)  node {$\bullet$} node[midway,below] {$2$};
\draw ($(A)+6*(v0)$)--++ (v1)  node {$\bullet$} --++ (v5) --++ (v1)  node {$\bullet$}  --++ (v5) --++ (v5) --++ (v4);
\draw ($(A)+6*(v0)$)--++ (v2)  node {$\bullet$} --++ (v0) node[midway, above] {$n-2$} --++ (v0) node[midway, above] {$n-1$} --++ (v0) node {$\bullet$}  node[midway, above] {$n$} --++ (v4) node[midway,right] {$c$};
\draw[very thick] ($(A)+(v0)$) --++ (v0) --++ (v1)--++(v3) --++ (v4);
\draw[very thick] ($(A)+7*(v0)$) --++ (v0) --++ (v2)--++(v3) --++ (v5);

\draw  (B) node {$\bullet$} --++ (v0) node {$\bullet$} node[midway, above] {$a$} --++ (vv0) node {$\bullet$} --++ (v0) node {$\bullet$}--++ (v2) node {$\bullet$} --++ (v4) --++ (vv2) node {$\bullet$} --++ (v4) node[midway,left] {$e$} --++ (v4) node {$\bullet$} --++ (v2) node[midway,left] {$b$}; 
\draw ($(B) + (v5)$) --++ (v4) node {$\bullet$} node[midway,left] {$c$}  --++ (v0)  node {$\bullet$} node[midway,below] {$a$} --++ (v0) node {$\bullet$}  node[midway,below] {$n$} --++ (vvv0) node {$\bullet$} node[midway,below] {$n-1$};
\draw ($(B) + (v5)$) --++ (v5);
\draw ($(B)+ (v1)+ (v0)$) --++ (vv0) node[midway, above] {$1$} --++ (v0) node {$\bullet$} node[midway, above] {$2$} --++ (v4);
\draw[dashed] ($(B)+ 2*(v0)+(vv0)$) --++ ($3*(v0)$);
\draw[dashed] ($(B)+ 2*(v0)+(vvv0)+ (v5)+(v4)$) --++ ($3*(v0)$);
\draw[dashed] ($(B)+3*(v0)+(vv0)$)--++ (v2) --++ (v0) --++ (v0) --++ (v4) --++ (v2)--++ (v4);
\draw ($(B)+5*(v0)+(vv0)$) node {$\bullet$} --++ (v0)  node {$\bullet$} --++ (vvv0)  node {$\bullet$} --++ (v0)  node {$\bullet$} node[midway,above] {$d$} --++ (v4)  node {$\bullet$}  node[midway,right] {$e$}--++ (v5)  node {$\bullet$}  node[midway,right] {$b$}--++ (v3)  node {$\bullet$} node[midway,below] {$d$} --++ (vv3)  node {$\bullet$} node[midway,below] {$1$}--++ (v3)  node {$\bullet$} node[midway,below] {$2$};
\draw ($(B)+5*(v0)+(vv0)$)--++ (v1)  node {$\bullet$} --++ (v5) --++ (vvv1)  node {$\bullet$}  --++ (v5) --++ (v5) --++ (v4);
\draw ($(B)+5*(v0)+(vv0)$)--++ (v2)  node {$\bullet$} --++ (v0) node[midway, above] {$n-2$} --++ (vvv0) node[midway, above] {$n-1$} --++ (v0) node {$\bullet$}  node[midway, above] {$n$} --++ (v4) node[midway,right] {$c$};
\draw[very thick] ($(B)+(v0)$) --++ (vv0) --++ (v1)--++(vv3) --++ (v4);
\draw[very thick] ($(B)+6*(v0)+(vv0)$) --++ (vvv0) --++ (v2)--++(vvv3) --++ (v5);

\draw[very thick,->]   (4.5,-2.5) --++ (0,-1);

\end{tikzpicture}

\caption{Example of a rigid surface that is not a local maximum}
\label{example:rigid:not:maximal}
\end{figure}

\begin{rem}\label{rem:true:singularities}
Note that the translation surface given in Figure~\ref{example:rigid:not:maximal} contains marked points in the set of singularities. We can easily make them true singularities by surgeries analogous to the ones described in the proof of Lemma~\ref{lem:constr:loc}.
\end{rem}

\begin{proof}
The fact that the surface is rigid is clear: when cut along shortest saddle connections it decomposes into equilateral triangles and a non self-adjacent polygon with no singularities in the interior in such a way that the set of triangles is connected. 

Now, we deform the surface as shown in the figure: the only short saddle connections that change are the horizontal ones in the parallelograms drawn with fat sides (see the labels ``1'' and ``$n-1$'') and their diagonals. The affine holonomy of the saddle connection corresponding to the label ``1'' is changed by adding $-i\varepsilon$ and similarly, we add $i\varepsilon$ to the one corresponding to the label ``$n-1$''. 

Since all short saddle connections keep to be of length at least one, we need to check that the area of the surface decreases. 
\begin{enumerate}
\item The area of each fat parallelogram increases exactly by the area of the gray parallelogram in Figure~\ref{parallelograms}, which is less that $\varepsilon$, and the two fat parallelograms in Figure~\ref{example:rigid:not:maximal} are disjoints for $n\geq 3$.
\begin{figure}[htb]

\begin{tikzpicture}[scale=2]
\coordinate (v0) at (1,0);
\coordinate (v1) at (0.5,0.866);
\coordinate (v3) at ($-1*(v0)$);
\coordinate (v4) at ($-1*(v1)$);
\coordinate (vv0) at (1,-0.3);
\coordinate (vv1) at ($(v1)+(0,-0.3)$); 
\coordinate (vv3) at ($-1*(vv0)$);

\fill[black,fill=gray!25] ($(vv0)+(v1)$) --++ (v4) --++ (0,0.3) --++ (v1) --cycle;

\draw (0,0) --++ (v0) --++ (v1)--++(v3) node[midway, above] {length=1} --++ (v4);
\draw[dashed] (0,0) --++ (vv0) --++ (v1) --++ (vv3) --++ (v4);
\draw[<->] ($(vv0)+(v1)+(0.3,0)$) --++ (0,0.3) node[midway,right] {$\varepsilon$};

\end{tikzpicture}

\caption{Comparing the area of the two parallelograms}
\label{parallelograms}

\end{figure}

\item The area of the polygon decreases by $(n-1)\varepsilon+ (n-2)\varepsilon=(2n-3)\varepsilon$.
\end{enumerate}
Hence the total area decreases if $n\geq 3$.
\end{proof}

\end{document}